\documentclass[a4paper,10pt]{article}
\usepackage{geometry}                
\usepackage[active]{srcltx}
\usepackage{hyperref}
\usepackage[normalem]{ulem}

\newcommand{\mms}{m.m.s.}

\usepackage{graphicx}
\usepackage{amssymb}
\usepackage{epstopdf}
\usepackage{amsmath}
\usepackage{amsthm}
\usepackage{amssymb}
\usepackage{latexsym}
\usepackage{mathtools}
\usepackage{mathrsfs}
\usepackage{graphics}
\usepackage{latexsym}
\usepackage{psfrag}
\usepackage{import}
\usepackage{verbatim}
\usepackage{enumerate}
\usepackage{enumitem}
\usepackage{graphicx}
\usepackage[usenames]{color}

\newcommand{\R}{\mathbb{R}}
\newcommand{\N}{\mathbb{N}}

\newcommand{\supp}{\text{\rm supp}}

\newcommand{\diam}{{\rm{diam\,}}}

\newcommand{\ve}{\varepsilon}

\renewcommand{\L}{\mathcal{L}}

\newcommand{\CD}{\mathsf{CD}}

\newcommand{\Geo}{{\rm Geo}}

\newcommand{\MCP}{\mathsf{MCP}}

\newcommand{\I}{\mathcal{I}}
\renewcommand{\L}{\mathcal{L}}

\newcommand{\vol}{\mathit{Vol}}

\renewcommand{\P}{\mathbb P}

\renewcommand{\P}{\mathcal{P}}

\renewcommand{\H}{\mathcal{H}}

\newcommand{\mm}{\mathfrak m}
\newcommand{\qq}{\mathfrak q}

\newcommand{\ee}{{\rm e}}

\newcommand{\sfd}{\mathsf d}

\newcommand{\Opt}{\mathrm{OptGeo}}

\newcommand{\AVR}{\mathsf{AVR}}

\theoremstyle{plain}
\newtheorem{lemma}{Lemma}[section]
\newtheorem{theorem}[lemma]{Theorem}

\newtheorem{proposition}[lemma]{Proposition}

\newtheorem*{theorem*}{Theorem}
\newtheorem*{maintheorem*}{Main Theorem}

\theoremstyle{definition}

\newtheorem{definition}[lemma]{Definition}
\newtheorem*{definition*}{Definition}
\newtheorem*{remark*}{Remark}

\numberwithin{equation}{section}

\title{Isoperimetric inequality in noncompact $\MCP$ spaces}
\author{Fabio Cavalletti\thanks{F. Cavalletti: Mathematics Area, SISSA, Trieste (Italy), email:cavallet@sissa.it. } \ and Davide Manini\thanks{D. Manini: Mathematics Area, SISSA, Trieste (Italy), email:dmanini@sissa.it.}}

\date{}     

\begin{document}
\maketitle

\begin{abstract}
We prove a sharp isoperimetric inequality 
for the class of metric measure spaces verifying the 
synthetic Ricci curvature lower bounds $\MCP(0,N)$ 
and having Euclidean volume growth at infinity. 
We avoid the classical use of the Brunn-Minkowski inequality,  
not available for $\MCP(0,N)$, and of the PDE approach, not available in the singular setting. Our approach will be carried over by using a scaling limit of localization.
\end{abstract}

\bibliographystyle{plain}


\section{Introduction}

The scope of this short note is to present an isoperimetric inequality 
for the class of metric measure spaces (or \mms\ for short)
having non-negative Ricci curvature and dimension bounded from above
encoded in the synthetic condition  
called Measure-Contraction property, for short $\MCP(0,N)$, where $N$ 
mimics the upper bound on the dimension.

The main motivation comes from the recent interest in functional and geometric inequalities in non-compact Riemannian manifolds having non-negative Ricci curvature. 
To mention the main contributions we list 
\cite{Brendle21,Brendle21b, Johne}
and \cite{AgFoMaz20, FogagnoloMazzieri}.

However the most general result concerning the isoperimetric inequality 
in the non-negative Ricci curvature setting is the one obtained in \cite{BaloghKristal:isoperimetric}. 
To properly report it, we first briefly introduce few notations. 
Let $(X,\sfd,\mm)$ be a metric-measure space (or \mms), meaning that $(X,\sfd)$ is a 
complete and separable metric space and $\mm$ a Radon, non-negative measure over it.   
For any real number $N > 1$ one can consider the family of those \mms's verifying 
the so-called Curvature-Dimension condition $\CD(0,N)$; this is 
 the synthetic condition introduced in the seminal papers \cite{lottvillani:metric,sturm:I,sturm:II} by Lott, Sturm and Villani. 
 We refer to these references for its definition; for the scope of this note 
 we only report that the $\CD(0,N)$ condition replaces in the non-smooth setting the two conditions $Ric \geq 0$ and the $\dim \leq N$.

Letting $B_{r}(x)=\{ y \in X \colon \sfd(x,y) < r\}$
denoting the metric ball with center $x\in X$
and radius $r> 0$, by Bishop-Gromov volume growth inequality, 
see \cite[Theorem 2.3]{sturm:II}, the map $r \mapsto \frac{\mm(B_{r}(x))}{r^{N}}$ 
is non-increasing over $(0,\infty)$  for any $x \in X$.
The \emph{asymptotic volume ratio} is then naturally defined by 
$$
\mathsf{AVR}_{(X,\sfd,\mm)} = \lim_{r\to\infty} \frac{\mm(B_{r}(x))}{\omega_{N} r^{N}}.
$$
It is easy to see that it is indeed independent of the choice of $x \in X$;
the constant $\omega_{N}$ is the volume of the Euclidean unit ball in $\R^{N}$ 
whenever $N \in\N$ and it is extended to real values of $N$ via the 
$\Gamma$ function. When $\AVR_{(X,\sfd,\mm)} > 0$, we say that $(X,\sfd,\mm)$ has Euclidean 
volume growth. 
Whenever no ambiguity is possible, we will prefer the shorter notation $\AVR_{X}$.
Even though in the notation  $\AVR_{(X,\sfd,\mm)}$ the dimension upper bound $N$ 
does not appear, its dependence on the asymptotic volume ratio has to be noticed.

The following sharp isoperimetric inequality has been obtained in \cite{BaloghKristal:isoperimetric}.

\begin{theorem}\label{T:BalKri}\cite[Theorem 1.1]{BaloghKristal:isoperimetric}
Let $(X,\sfd,\mm)$ be a \mms\ satisfying the $\CD(0,N)$ condition for some $N > 1$, 
and having Euclidean volume growth. 
Then for every bounded Borel subset $E \subset X$ it holds
\begin{equation}\label{E:inequalityBalKri}
\mm^{+}(E) \geq N \omega_{N}^{\frac{1}{N}} \AVR_{(X,\sfd,\mm)}^{\frac{1}{N}} 
\mm(E)^{\frac{N-1}{N}},
\end{equation}
where $\mm^{+}(E)$ denotes the outer Minkowski content of $E$.
Moreover, inequality \eqref{E:inequalityBalKri} is sharp.
\end{theorem}

The proof of Theorem \ref{T:BalKri} follows from 
the refined Brunn-Minkowski inequality given by optimal transport.

More challenging to prove, \eqref{E:inequalityBalKri} also enjoys rigidity properties once restricted to the smooth world. 
If $(M,g)$ is a noncompact, complete $n$-dimensional 
Riemannian manifold having nonnegative Ricci curvature, 
one can consider $(M,d_{g},\vol_{g})$ as a metric measure space, where $d_{g}$ and 
$\vol_{g}$ denote the natural metric and canonical measure on $(M,g)$, respectively, 
and $(M,d_{g},\vol_{g})$ verifies the $\CD(0,n)$ condition.
The asymptotic volume ratio of $(M,g)$ is then given by 
$\AVR_{g} := \AVR_{(M,d_{g} ,\vol_{g})}$. 
By the Bishop-Gromov theorem one has that $\AVR_{g} \leq 1$ with
$\AVR_{g} = 1$ if and only if $(M, g)$ is isometric to the usual Euclidean space 
$\R^{n}$ endowed with the Euclidean metric $g_{0}$.

The rigidity result proved in  \cite[Theorem 1.2]{BaloghKristal:isoperimetric}
states that the equality holds in \eqref{E:inequalityBalKri} 
for some $E\subset M$ with $C^{1}$ smooth regular boundary and $M$ smooth manifold 
if and only if 
$\AVR_{g} = 1$ and $E$ is isometric to a ball $B \subset \R^{n}$.

Our scope is to extend the validity of Theorem 
\ref{T:BalKri} to a wider class of spaces. 
In particular our interest is in the class of non-smooth 
\mms\ verifying the Measure-Contraction property $\MCP(0,N)$, 
a synthetic curvature-dimension condition strictly weaker  than the $\CD(0,N)$ condition. The $\MCP$ condition has been introduced independently in the two contributions \cite{Ohta1} and \cite{sturm:II} and it implies the Bishop-Gromov 
volume comparison theorem, making the asymptotic volume ratio meaningful also for this class of spaces. 
Being a weaker condition, 
the $\MCP(0,N)$ has been verified to hold true for a large family of sub-Riemannian spaces for which no curvature-dimension conditions hold true (see \cite{BarRiz19}).
A sharp isoperimetric inequality, to the best of our knowledge, is completely open in the sub-Riemannian setting; even for the Heisenberg group no sharp isoperimetric inequality is known.  
This furnishes our main motivation to extend Theorem \ref{T:BalKri}
to $\MCP$ spaces.

\smallskip
From the technical point of view, the main issue with the isoperimetric inequality in 
the non-compact $\MCP$ setting is the absence of a classical Brunn-Minkowski inequality, crucially used in some recent contributions (see for instance \cite{BaloghKristal:isoperimetric}).  
It has to be underlined that recently a \emph{modified} Brunn-Minkowski inequality 
has been established in \cite{BaloghKristal:Heisenberg, BarRiz19} 
for a large family of sub-Riemannian manifolds also verifying  the $\MCP(0,N)$
condition, for an appropriate choice of $N>1$. 
Due to the nonlinearity of the concavity interpolation coefficients, 
this  modified version of 
Brunn-Minkowski inequality seems to not directly imply a non-trivial isoperimetric inequality. 
Also the weaker versions of the 
Brunn-Minkowski inequality obtained in \cite{MilQCD} (see Section 7), verified 
again by a large family of  sub-Riemannian spaces,
are just not tailored to obtain an expansion of the volume of a tubular neighbourhood 
of a given set.

\smallskip
We will solve this problem relying on an asymptotic use of 
the localization paradigm. The price we have to pay to use a 
more sophisticated approach compared 
to the direct use of the Brunn-Minkowski inequality, 
is the essentially non-branching assumption, 
a classical hypothesis that will be discussed in Section \ref{Ss:local}. 
The following one is the main result of this note.

\begin{theorem}\label{T:main1}
Let $(X,\sfd,\mm)$ be an essentially non-branching \mms\ satisfying the $\MCP(0,N)$ condition for some $N > 1$, 
and having Euclidean volume growth. 
Then for every Borel subset $E \subset X$ with $\mm(E)<\infty$, it holds
\begin{equation}\label{E:inequality}
\mm^{+}(E) \geq  (N\omega_{N} \AVR_{(X,\sfd,\mm)})^{\frac{1}{N}} 
\mm(E)^{\frac{N-1}{N}}.
\end{equation}
Moreover, inequality \eqref{E:inequality} is sharp.
\end{theorem}

Some comments on Theorem \ref{T:main1} are in order. \\
Notice the slightly worse constant of \eqref{E:inequality} 
compared to \eqref{E:inequalityBalKri}.
The rigidity side of Theorem \ref{T:main1}, due to the possibly nonlinear 
infinitesimal structure of the space, seems to be at the moment out of reach. 
Notice for instance that no metric Splitting Theorem is available for $\MCP(0,N)$ 
spaces (for the weaker measure-theoretic splitting in this setting we 
refer to \cite[Theorem 7.1]{CM18}). 

Finally, let us mention that Theorem \ref{T:main1} will not imply a non-trivial isoperimetric inequality in the Heisenberg group of any dimension. Indeed while for instance 
the first Heisenberg group satisfies $\MCP(0,5)$, the volume of its geodesic balls grows with the fourth power, giving zero $\AVR$.
For the details we refer to \cite{BarRiz19} and to \cite{MilQCD}.

The paper is organised as follows: in Section \ref{S:preliminaries} 
we recall the basics for $\MCP$, localization  and the 
isoperimetric inequality in these spaces. 
Section \ref{S:inequality} contains the proof of Theorem \ref{T:main1}.

\section{Preliminaries}\label{S:preliminaries}
A triple $(X,\sfd,\mm)$  is called  metric measure space (or m.m.s.)  if $(X,\sfd)$ a Polish space (i.e. a  complete and separable metric space) and $\mm$ is  a positive Radon measure over $X$.  

We will briefly recall the basic notations to give the definition of the $\MCP(0,N)$.
The space of constant speed geodesics is denoted 
$$
\Geo(X) : = \{ \gamma \in C([0,1], X):  \sfd(\gamma_{s},\gamma_{t}) = |s-t| \sfd(\gamma_{0},\gamma_{1}), \text{ for every } s,t \in [0,1] \},
$$
and the metric space $(X,\sfd)$ is a geodesic space if for each $x,y \in X$ 
there exists $\gamma \in \Geo(X)$ so that $\gamma_{0} =x, \gamma_{1} = y$.


$\mathcal{P}(X)$ is the  space of Borel probability measures over $X$ and 
$\mathcal{P}_{2}(X)$ the subset of those measures with finite second moment; on
$\mathcal{P}_{2}(X)$ one defines the $L^{2}$-Wasserstein distance  $W_{2}$:  
for $\mu_0,\mu_1 \in \mathcal{P}_{2}(X)$,
$$
  W_2^2(\mu_0,\mu_1) := \inf_{ \pi} \int_{X\times X} \sfd^2(x,y) \, \pi(dxdy),
$$
where the infimum is taken over all $\pi \in \mathcal{P}(X \times X)$ with $\mu_0$ and $\mu_1$ as the first and the second marginals.
The space $(X,\sfd)$ is geodesic  if and only if the space $(\mathcal{P}_2(X), W_2)$ is geodesic.  For any $t\in [0,1]$,  let ${\rm e}_{t}$ denote the evaluation map: 
$$
  {\rm e}_{t} : \Geo(X) \to X, \qquad {\rm e}_{t}(\gamma) : = \gamma_{t}.
$$
Any geodesic $(\mu_t)_{t \in [0,1]}$ in $(\mathcal{P}_2(X), W_2)$  can be lifted to a measure $\nu \in {\mathcal {P}}(\Geo(X))$, 
so that $({\rm e}_t)_\sharp \, \nu = \mu_t$ for all $t \in [0,1]$, where as usual 
the subscript $\sharp$ denotes the push-forward.

Given $\mu_{0},\mu_{1} \in \mathcal{P}_{2}(X)$, we denote by 
$\Opt(\mu_{0},\mu_{1})$ the space of all $\nu \in \mathcal{P}(\Geo(X))$ for which $({\rm e}_0,{\rm e}_1)_\sharp\, \nu$ 
realizes the minimum in the definition of $W_{2}$. Such a $\nu$ will be called \emph{dynamical optimal plan}. If $(X,\sfd)$ is geodesic, then the set  $\Opt(\mu_{0},\mu_{1})$ is non-empty for any $\mu_0,\mu_1\in \mathcal{P}_2(X)$.
Finally $\mathcal{P}_{2}(X,\sfd,\mm)\subset \mathcal{P}_{2}(X)$
is the subspace formed by all those measures absolutely continuous with respect to $\mm$.


\subsection{Measure-Contraction Property}

We briefly describe the $\MCP$ condition,  
a synthetic way to impose a
Ricci curvature lower bound and 
an upper bound on the dimension.  To keep the presentation as simple as possible 
we will confine ourselves to the case $K = 0$.
The interested reader can consult \cite{Ohta1}.


\begin{definition}[$\MCP(0,N)$] \label{D:Ohta1}
A m.m.s. $(X,\sfd,\mm)$ is said to satisfy $\MCP(0,N)$ if for any $o \in \supp(\mm)$ and  $\mu_0 \in \P_2(X,\sfd,\mm)$ of the form $\mu_0 = \frac{1}{\mm(A)} \mm\llcorner_{A}$ for some Borel set $A \subset X$ with $0 < \mm(A) < \infty$, 
there exists $\nu \in \Opt(\mu_0, \delta_{o} )$ such that:
\begin{equation} \label{eq:MCP-def}
\frac{1}{\mm(A)} \mm \geq (\ee_{t})_{\sharp} \big(  (1-t)^{N} \nu(d \gamma) \big) \;\;\; \forall t \in [0,1] .
\end{equation}
\end{definition}

If $(X,\sfd,\mm)$ is a m.m.s. verifying $\MCP(K,N)$, no matter for which $K \in \R$, then $(\supp(\mm),\sfd)$  is Polish, proper and it is a geodesic space. 
With no loss in generality for our purposes we will assume that $X = \supp(\mm)$. 


To conclude, we report the following important fact
\cite[Theorem 3.2]{Ohta1}: if $(M,g)$ is $n$-dimensional Riemannian manifold 
with $n\geq 2$, the m.m.s. $(M,d_{g},\vol_{g})$ verifies $\MCP(K,n)$ if and only if $Ric_{g} \geq K g$, where $d_{g}$ is the geodesic distance induced by $g$ and $\vol_{g}$ the volume measure.

We refer to \cite{Ohta1, sturm:II} for more general results.
For the other synthetic versions of Ricci curvature lower bounds 
we refer to \cite{lottvillani:metric,sturm:I,sturm:II}.


\medskip
\subsection{Isoperimetry in normalised \texorpdfstring{$\MCP$}{MCP} spaces}\label{Ss:IsopMCP}

Isoperimetric inequalities relate the size of the boundary of a set to its volume.  
A handy way to measure the boundaries is the outer Minkowski content: for 
any Borel set $E \subset X$ having $\mm(E) < \infty$
$$
\mm^{+}(E) : = \liminf_{\ve \to 0} \frac{\mm(E^{\ve}) - \mm(E)}{\ve}, \qquad E^{\ve} 
: = \{ y \in X \colon \sfd(y,E) \leq \ve \}; 
$$
the isoperimetric profile function of a \mms\ $(X,\sfd,\mm)$ 
is defined as follows:
$$
\mathcal{I}_{(X,\sfd,\mm)}(v) : = \inf \{ \mm^{+}(E) \colon \mm(E) = v \}. 
$$

A sharp isoperimetric inequality \`a la Levy-Gromov for the class of spaces verifying 
$\MCP(K,N)$ with finite diameter and total mass 1 is the main content of \cite{CS19}.

\begin{theorem}[\cite{CS19}]\label{T:ISOMCP}
Let $K,N,D\in \R$ with $N>1$ and $D>0$. Then there exists an explicit non-negative function 
$\I_{K,N,D}^{\MCP} : [0,1] \to \R$ such that the following holds. 

If $(X,\sfd,\mm)$ is an essentially non-branching m.m.s. verifying $\MCP(K,N)$ with 
$\mm(X) = 1$ and having diameter less than $D$  and $A\subset X$, then
\begin{equation}\label{E:ISOMCP}
\mm^{+}(A) \geq \I_{K,N,D}^{\MCP}(\mm(A)).
\end{equation}
Moreover \eqref{E:ISOMCP} is sharp, i.e. for each $v \in [0,1]$, $K,N,D$ there exists 
a m.m.s. $(X,\sfd,\mm)$ with $\mm(X)=1$ and $A\subset X$ 
with $\mm(A) = v$ such that \eqref{E:ISOMCP} is an equality. 
\end{theorem}

The sharp lower bound on the isoperimetric profile function \eqref{E:ISOMCP}
has an explicit expression. We report only the case $K = 0$:  for each $v \in [0,1]$
\begin{equation}
\label{equ:sharp1d}
\I_{0,N,D}^{\MCP}(v)=
f_{0,N,D} (a_{0,N,D}(v)),
\end{equation}
where the function $f_{0,N,D}$ is defined in the following way
$$
 f_{0,N,D}(x) : = \left( 
\int_{(0,x)}\left( \frac{D -  y}{D - x} \right)^{N-1}  dy 
+\int_{(x,D)} \left( \frac{y}{x} \right)^{N-1} dy \right)^{-1},
$$
and the function $a_{0,N,D}$ is obtained as follows:
define the function 
\begin{equation}\label{E:anotherway}
v_{0,N,D}(a) = \frac{f_{0,N,D}(a)}{(D-a)^{N-1}} \int_{(0,a)}
(D- x)^{N-1} \, dx
=
f_{0,N,D}(a)
\frac{D^N-(D-a)^N}{N(D-a)^{N-1}};
\end{equation}
as proved in \cite{CS19}, for each $N,D$ 
it is possible to define the inverse map of $v$:
\begin{equation*}
[0,1] \ni v \longmapsto a_{0,N,D}(v) \in (0,D); 
\end{equation*}
hence we have recalled the definition of each function used in \eqref{equ:sharp1d} 
to construct the lower bound $\I_{0,N,D}^{\MCP}$.

For Ricci lower bounds other than 0, we refer to \cite{CS19}.

\smallskip

\subsection{Localization in compact \texorpdfstring{$\MCP$}{MCP} spaces}
\label{Ss:local}

The localization method reduces the task of establishing various analytic and geometric inequalities on a higher dimensional space to the one-dimensional setting. 
It has been established in numerous settings that we now shortly review. 

In the Euclidean setting, the approach has its roots in the work of Payne and Weinberger \cite{PayneWeinberger} on the spectral gap for convex domains, 
and has been further developed and popularised by Gromov and V. Milman \cite{Gromov-Milman} and Kannan, Lovasz and Simonovits \cite{KLS}. 
In a ground-breaking work in 2015, Klartag \cite{klartag} reinterpreted the localization method as a measure disintegration adapted to 
$L^{1}$-Optimal-Transport, 
and extended it to weighted Riemannian manifolds satisfying $\CD(K,N)$. 

Subsequently, the first author and Mondino \cite{CM1} have succeeded to 
extend this technique to the non-smooth framework of essentially non-branching geodesic \mms's $(X,\sfd,\mm)$ verifying $\CD(K,N)$, $N \in (1,\infty)$. 
In the non-smooth world are also worth  mentioning the generalisation to the Finsler setting by Ohta \cite{Ohta2} and to the Lorentzian length spaces by Cavalletti–Mondino \cite{CM20}.

Localization for $\MCP(K,N)$ was, partially and in a different form, already known 
in 2009, see \cite[Theorem 9.5]{biacava:streconv}, for non-branching \mms's. 
The case of essentially non-branching \mms's and the effective reformulation after the work of Klartag \cite{klartag} has been recently discussed in 
\cite[Section 3]{CM18} to which we refer for all the missing details (see in particular \cite[Theorem 3.5]{CM18}).

\smallskip
To properly formulate the version of localization we will adopt, 
we recall the definition 
of essentially non-branching spaces, introduced in \cite{RS2014}. 
A set $G \subset \Geo(X)$ is a set of non-branching geodesics if and only if for any $\gamma^{1},\gamma^{2} \in G$, it holds:
$$
\exists \;  \bar t\in (0,1) \text{ such that } \ \forall t \in [0, \bar t\,] \quad  \gamma_{ t}^{1} = \gamma_{t}^{2}   
\quad 
\Longrightarrow 
\quad 
\gamma^{1}_{s} = \gamma^{2}_{s}, \quad \forall s \in [0,1].
$$
\begin{definition}\label{def:ENB}
A metric measure space $(X,\sfd, \mm)$ is \emph{essentially non-branching} if and only if for any $\mu_{0},\mu_{1} \in \mathcal{P}_{2}(X)$,
with $\mu_{0},\mu_{1}$ absolutely continuous with respect to $\mm$, any element of $\Opt(\mu_{0},\mu_{1})$ is concentrated on a set of non-branching geodesics.
\end{definition}
Clearly the space itself is non-branching if $\Geo(X)$ is a set of non-branching geodesics.

Before stating the localization theorem, we recall few notations: 
we denote by $\H^{1}$ the one-\-dimen\-sion\-al Hausdorff measure on the underlying metric space and by $\mathcal{M} (X,\mathcal{X})$ all the non-negative measures over 
the measurable space $(X,\mathcal{X})$.

Consider a measure space $(X, \mathcal{X} ,\mm)$ and 
$\{X_{\alpha}\}_{\alpha \in Q}$ a partition  of $X$.
A disintegration of
$\mm$ on $\{X_{\alpha}\}_{\alpha \in Q}$ is a measure space structure
$(Q,\mathcal{Q},\qq)$ and a map
$$
Q \ni \alpha \mapsto \mm_{\alpha} \in \mathcal{M} (X,\mathcal{X}),
$$
such that 
\begin{itemize}
\item For $\qq$-a.e. $\alpha \in Q$, $\mm_{\alpha}$ is concentrated on $X_{\alpha}$.
\item For all $B \in \mathcal{X}$ , the map $\alpha \mapsto \mm_{\alpha}(B)$ is 
$\qq$-measurable.
\item For all $B \in \mathcal{X}$, $\mm(B) = \int_{Q} \mm_{\alpha}(B)\,\qq(d\alpha)$; this is abbreviated by $\mm = \int_{Q} \mm_{\alpha}\,\qq(d\alpha)$.
\end{itemize}

\smallskip

\begin{theorem}[Localization on $\MCP(K,N)$ spaces, {\cite[Theorem 3.5]{CM18}}]\label{T:locMCP}
Let $(X,\sfd,\mm)$ be an essentially non-branching m.m.s. with $\supp(\mm) = X$ and satisfying $\MCP(K,N)$, for some $K\in \R, N\in (1,\infty)$.

Let $g : X \to R$ be $\mm$-integrable with $\int_{X} g\,\mm = 0$ and 
$\int_{X} |g(x)|\sfd(x,x_{0})\mm(dx) < \infty$ for some $x_{0} \in X$. 
Then there exists an $\mm$-measurable subset $\mathcal{T} \subset X$ 
and a family $\{X_{\alpha}\}_{\alpha \in Q}$ of subsets of $X$, such that there exists a disintegration of $\mm\llcorner_{\mathcal{T}}$ on 
$\{X_{\alpha}\}_{\alpha \in Q}$:
$$
\mm\llcorner_{\mathcal{T}}= \int_{Q} \mm_{\alpha}\,\qq(d\alpha),\qquad \qq(Q)=1,
$$
and for $\qq$-a.e. $\alpha \in Q$:
\begin{enumerate}
\item  $X_{\alpha}$ is a closed geodesic in $(X, \sfd)$.
\item 
$\mm_{\alpha}$ is a Radon measure supported on $X_{\alpha}$ with 
$\mm_{\alpha} \ll \H^{1}\llcorner_{X_{\alpha}}$. 
\item  The metric measure space 
$(X_{\alpha} , \sfd, \mm_{\alpha} )$ verifies $\MCP(K, N )$. 
\item $\int g\, d\mm_{\alpha} = 0$, 
and $g = 0$ $\mm$-a.e. on $X \setminus \mathcal{T}$.
\end{enumerate}
Finally, the $X_{\alpha}$ are called transport rays and two distinct transport rays 
can only meet at their extremal points.
\end{theorem}

It is worth stressing that 
the restriction to essentially non-branching m.m.s.'s is done to avoid pathological cases:
as an example of possible pathological behaviour we mention the failure of the local-to-global property of $\CD(K,N)$ within this class of spaces; in particular, a heavily-branching m.m.s. verifying $\CD_{loc}(0,4)$ which does not verify $\CD(K,N)$ for any fixed $K \in \R$ and $N \in [1,\infty]$ was constructed by Rajala in \cite{R2016}, while the local-to-global property of $\CD(K,N)$
has been recently proved to hold \cite{CMi16} for essentially non-branching m.m.s.'s.

\smallskip
As localization method suggests, 
a relevant class of spaces needed for our purposes  is the one of one-dimensional spaces $(X,\sfd,\mm)=(I,|\cdot|, h \L^1)$.
It is a standard fact that the m.m.s. $(I,|\cdot|,h\L^{1})$ verifies $\MCP(0,N)$ if and only if the non-negative Borel function $h$ satisfies the following inequality: 
\begin{equation}\label{E:MCPdef}
 h(t x_1 + (1-t) x_0) \geq (1-t)^{N-1} h(x_0),
\end{equation}
for all $x_0,x_1 \in I$ and $t \in [0,1]$, see for instance 
\cite[Theorem 9.5]{biacava:streconv} where also the case $K \neq 0$ is discussed. 
We will call $h$ an $\MCP(0,N)$-density.

\smallskip
Inequality \eqref{E:MCPdef} implies several known properties that we recall for readers' convenience.
If we confine ourselves to the case $I = (a,b)$ with $a,b \in \R$
\eqref{E:MCPdef} implies (actually  is equivalent to)
\begin{equation}\label{E:MCPdef2}
\left( \frac{b - x_{1} }{b - x_{0}} \right)^{N-1} 
\leq \frac{h(x_{1} ) }{h (x_{0})} \leq 
\left( \frac{ x_{1} -a  }{ x_{0} -a } \right)^{N-1}, 
\end{equation}
for $x_{0} \leq x_{1}$. 
In particular, if we consider the unbounded case $I=[0,+\infty)$,
  \eqref{E:MCPdef} is equivalent to
\begin{equation}\label{E:MCPdef3}
1
\leq \frac{h(x_{1} ) }{h (x_{0})} \leq 
\left( \frac{x_{1}}{x_{0}} \right)^{N-1}.
\end{equation}
%
In both cases, $h$ is locally Lipschitz in the interior of $I$ and continuous up to the boundary.
%
%
%

\bigskip
\section{Sharp isoperimetric inequalities in \texorpdfstring{$\MCP(0,N)$}{MCP(0,N)}}\label{S:inequality}

To prove Theorem \ref{T:main1} we will need to consider the isoperimetric problem 
inside a family of large subsets of $X$ with diameter approaching $\infty$. 
In order to apply the classical dimension reduction argument furnished by localization theorem (Theorem \ref{T:locMCP}), one needs in principle these subsets to also be convex. As the existence of an increasing family of convex subsets recovering at the limit the whole space $X$ is in general false,   
we will overcome this issue in the following way. 

Given any bounded set $E \subset X$ with $0< \mm(E) < \infty$, fix any point
 $x_{0} \in E$
 and  then consider $R > 0$ such that $E \subset B_{R}$
 (hereinafter we will adopt the following notation $B_R:=B_R(x_0)$).
Consider then the following family of zero mean
 functions: 
$$
g_{R} (x) = \left(\chi_{E} - \frac{\mm(E)}{\mm(B_{R})} \right)\chi_{B_{R}}.
$$ 
Clearly $g_{R}$ satisfies the hypothesis of Theorem \ref{T:locMCP} 
so we obtain an 
$\mm$-measurable subset $\mathcal{T}_{R} \subset X$ 
and a family $\{X_{\alpha,R}\}_{\alpha \in Q_{R}}$  of transport rays, such that there exists a disintegration of $\mm\llcorner_{\mathcal{T}_{R}}$ on 
$\{X_{\alpha,R}\}_{\alpha \in Q_{R}}$:
\begin{equation}\label{E:disintbasic}
\mm\llcorner_{\mathcal{T}_{R}}= \int_{Q_{R}} \mm_{\alpha,R}\,\qq_{R}(d\alpha),\qquad \qq_{R}(Q_{R})=1,
\end{equation}
with the Radon measures $\mm_{\alpha,R}$ having an $\MCP(0,N)$ density with respect to $\H^{1}\llcorner_{X_{\alpha,R}}$. The localization of the zero mean implies that 
$$
\mm_{\alpha,R}(E) = \frac{\mm(E)}{\mm(B_{R})} \mm_{\alpha,R}(B_{R}), \qquad 
\qq_{R}\text{-a.e.} \ \alpha \in Q_{R}.
$$
By using a unit speed parametrisation of the geodesic $X_{\alpha,R}$, 
without loss of generality we can assume that 
$\mm_{\alpha,R} = h_{\alpha,R} \mathcal{L}^{1}\llcorner_{[0,\ell_{\alpha,R}]}$, 
where $\ell_{\alpha,R}$ denotes the (possibly infinite) length of the $X_{\alpha,R}$.
Also we specify that the direction of the parametrisation of $X_{\alpha,R}$
is chosen such that $0 \in E$. Equivalently, if $u_{R}$ denotes a Kantorovich potential 
associated to the localization of $g_{R}$, then the parametrisation is chosen in such a way 
that $u_{R}$ is decreasing along $X_{\alpha,R}$ with slope $-1$.

Now we can define $T_{\alpha,R}$ to be the unique element of $[0,\ell_{\alpha,R}]$
such that %
$\mm_{\alpha,R}([0,T_{\alpha,R}]) = \mm_{\alpha,R}(B_{R})$. 
Notice that  $\diam(B_{R} \cap X_{\alpha,R}) \leq R + \diam(E)$: 
 if $\gamma$ is a unit speed parametrization of $X_{\alpha,R}$, then
$\sfd(\gamma_{0},\gamma_{t}) \leq \sfd(\gamma_{0},x_{0}) + \sfd(\gamma_{t},x_{0}) 
\leq \diam(E) + R$,
provided $\gamma_{t} \in B_{R}\cap X_{\alpha,R}$.  
Hence the same upper bound is valid for $T_{\alpha,R}$, i.e. $T_{\alpha,R}\leq R + \diam(E)$.

The plan will be to restrict $\mm_{\alpha,R}$ to 
${[0,T_{\alpha,R}]}$ so to have the following disintegration: 
\begin{equation}\label{E:disintnormalized}
\mm\llcorner_{\bar{\mathcal{T}}_{R}} = \int_{Q_{R}} \bar{\mm}_{\alpha,R}\, \bar{\qq}_{R}(d\alpha),
\end{equation}
where $\bar{\mm}_{\alpha,R} : =  \mm_{\alpha,R}\llcorner_{[0,T_{\alpha,R}]}/\mm_{\alpha,R}(B_{R})$ are probability measures, 
$\bar{\qq}_{R} = \mm_{\cdot,R}(B_{R}) \qq_{R}$ (in particular 
$\bar{\qq}_{R}(Q_{R}) = \mm(B_{R})$, using \eqref{E:disintbasic} and the fact that 
$B_{R}\subset \mathcal{T}_{R}$) and 
$\bar{\mathcal{T}}_{R} = \cup_{\alpha \in Q_{R}} [0,T_{\alpha,R}]$, where 
we are identifying $[0,T_{\alpha,R}]$ with the geodesic segment of length 
$T_{\alpha,R}$ of $X_{\alpha,R}$, that will be denoted by $\bar X_{\alpha,R}$

The disintegration \eqref{E:disintnormalized} will have applications 
only if $(E \cap X_{\alpha,R}) \subset [0,T_{\alpha,R}]$, implying that 
$$
\bar{\mm}_{\alpha,R}(E) = \frac{\mm(E)}{\mm(B_{R})}, 
\qquad 
\qq_{R}\text{-a.e.} \ \alpha \in Q_{R}.
$$
To prove this inclusion we will impose that $E \subset B_{R/4}$. 
If we denote by $\gamma^{\alpha,R} : [0,\ell_{\alpha,R}] \to X_{\alpha,R}$ the unit speed parametrisation, we notice that  
$$
\sfd(\gamma_{t}^{\alpha,R},x_{0}) \leq \sfd(\gamma_{0}^{\alpha,R},x_{0}) + t \leq \diam(E) + t \leq   \frac{R}{2} + t,
$$
where in the second inequality we have used that each starting point of the transport ray has to be inside $E$, being precisely where $g_{R} > 0$.
Hence $\gamma_{t}^{\alpha,R} \in B_{R}$ for all $t < R/2$. 
This implies that $(\gamma^{\alpha,R})^{-1}(B_{R}) \supset 
[0,\min\{R/2, \ell_{\alpha,R}\}]$, hence ``no holes'' inside $(\gamma^{\alpha,R})^{-1}(B_{R})$
before $\min\{R/2, \ell_{\alpha,R}\}$, implying that
$T_{\alpha,R} \geq \min\{R/2, \ell_{\alpha,R}\}$.

Since $\diam(E) \leq R/2$, necessarily $(\gamma^{\alpha,R})^{-1}(E) 
\subset [0,\min\{R/2,\ell_{\alpha,R}\}]$ implying that 
$(E\cap X_{\alpha,R}) \subset [0,T_{\alpha,R}]$. 
We summarise this construction in the following 

\begin{proposition}\label{P:disintfinal}
Given any bounded $E \subset X$ with $0< \mm(E) < \infty$, fix any point $x_{0} \in E$
and  then fix $R > 0$ such that $E \subset B_{R/4}(x_{0})$.  

Then there exists a Borel set $\bar{\mathcal{T}}_{R} \subset X$, with 
$E \subset \bar{\mathcal{T}}_{R}$ and a disintegration formula 
\begin{equation}\label{E:disintfinal}
\mm\llcorner_{\bar{\mathcal{T}}_{R}} 
= \int_{Q_{R}} \bar{\mm}_{\alpha,R}\, \bar{\qq}_{R}(d\alpha), \qquad 
\bar{\mm}_{\alpha,R}(\bar X_{\alpha,R}) = 1, \qquad \bar{\qq}_{R}(Q_{R}) = \mm(B_{R}),
\end{equation}
such that $\bar{\mm}_{\alpha,R}(E) = \frac{\mm(E)}{\mm(B_{R})}$, $\bar \qq_{R}$-a.e. and 
the one-dimensional \mms \ $(\bar X_{\alpha,R}, \sfd,\bar{\mm}_{\alpha,R})$ 
verifies $\MCP(0,N)$ and has diameter bounded by $R + \diam (E)$.
\end{proposition}

\subsection{Proof of the main inequality}

We now look for a simple expansion of $\I_{0,N,D}^{\MCP}(v)$ for 
$v$ close to $0$. 
%
%
%
%
%
For the definition of $\I_{0,N,D}^{\MCP}$ see Section~\ref{Ss:IsopMCP}.

We start with the case $D=1$ and we recall that
$$
f_{0,N,1} (x) = \left( \int_{(0,x)} \left(\frac{1-y}{1-x}\right)^{N-1}\,dy +
\int_{(x,1)} \left(\frac{y}{x}\right)^{N-1}\,dy \right)^{-1},
$$
obtaining 
\begin{align*}
f_{0,N,1} (x) 
&~ = \left( \frac{1 - (1-x)^{N}}{N(1-x)^{N-1}}  
+ \frac{1 - x^{N}}{Nx^{N-1}}\right)^{-1} = 
N
\left( \frac{1}{(1-x)^{N-1}} - 1 + \frac{1}{x^{N-1}}\right)^{-1} 
\\
&~ =  Nx^{N-1} 
\left( \left(\frac{x}{1-x}\right)^{N-1} - x^{N-1} +1 \right)^{-1} = Nx^{N-1}  + o(x^{N-1}).
\end{align*}
Then looking at $v_{0,N,1}(a)$
\begin{align*}
v_{0,N,1}(a) &~ = f_{0,N,1}(a) \frac{1 - (1-a)^{N}}{N(1-a)^{N-1}} \\
&~ = f_{0,N,1}(a)( a + o(a)) = N a^{N} + o(a^{N}),
\end{align*}
giving that 
$a_{0,N,1}(v) = N^{-\frac{1}{N}}v^{\frac{1}{N}} + o(v^{\frac{1}{N}})$
and implying that 
\begin{equation}\label{E:expansion1}
\I_{0,N,1}^{\MCP}(v) = N a_{0,N,1}(v)^{N-1} + o(a_{0,N,1}(v)^{N-1}) 
= N^{\frac{1}{N}} v^{\frac{N-1}{N}} + o(v^{\frac{N-1}{N}}).
\end{equation}

To obtain the general case when $D$ is arbitrary we use the following relation
proved in Lemma~3.9 of~\cite{CS19}: $\I_{0,N,1}^{\MCP} \leq D \I_{0,N,D}^{\MCP}$.
%
%
%
%
%
This means that~\eqref{E:expansion1} yields
\begin{equation}\label{E:expansion}
\I_{0,N,D}^{\MCP}(v) 
\geq \frac{1}{D}\left(N^{\frac{1}{N}} v^{\frac{N-1}{N}} + o(v^{\frac{N-1}{N}})\right).
\end{equation}

We now deduce Theorem \ref{T:main1} from the expansion of 
\eqref{E:expansion} and Theorem \ref{T:ISOMCP} applied to a family of one-dimensional 
spaces. The dimensional reduction argument is a classical application of the localization paradigm.

\begin{theorem}\label{T:isoperimetricAVR}
Let $(X,\sfd,\mm)$ be an essentially non-branching \mms \, verifying  $\MCP(0,N)$ and having $\AVR_{X} > 0$.
Let $E \subset X$ be any Borel set with $\mm(E)<\infty$, then
\begin{equation}\label{E:isopAVR}
\mm^{+}(E) \geq 
\left(N \omega_{N}\AVR_{X}\right)^{\frac{1}{N}} \mm(E)^{\frac{N-1}{N}}.
\end{equation}
\end{theorem}

\begin{proof}

{\bf Step 1.} Assume $E$ to be a bounded set. \\ 
Let $x_{0} \in E$ be any point and  
 consider $R > 0$ such that $E \subset B_{R/4}(x_{0})$; 
Proposition~\ref{P:disintfinal} implies we have the following disintegration 
$$
\mm\llcorner_{\bar{\mathcal{T}}_{R}} 
= \int_{Q_{R}} \bar{\mm}_{\alpha,R}\, \bar{\qq}_{R}(d\alpha), \qquad 
\bar{\mm}_{\alpha,R}(\bar X_{\alpha,R}) = 1, \qquad \bar{\qq}_{R}(Q_{R}) = \mm(B_{R}), 
$$
%
%
%
with $E\subset \bar{\mathcal{T}}_{R}$ and for $\bar{\qq}_{R}$-a.e. $\alpha \in Q_{R}$, 
$\bar{\mm}_{\alpha,R}(E) = \frac{\mm(E)}{\mm(B_{R})}$. 
At this point we can compute the outer Minkowski content of $E$: 
\begin{align*}
\mm^{+}(E) 
= &~ \liminf_{\ve \to 0} \frac{\mm(E^{\ve}) - \mm(E)}{\ve} 
\geq \liminf_{\ve \to 0} \frac{\mm(E^{\ve}\cap \bar{\mathcal{T}}_{R}) - \mm(E)}{\ve} 
\\
\geq &~ \int_{Q_{R}}\liminf_{\ve \to 0}  
\frac{\bar{\mm}_{\alpha,R}(E^{\ve}) - \bar{\mm}_{\alpha,R}(E)}{\ve}\, \bar{\qq}_{R}(d\alpha) \\
\geq &~ \int_{Q_{R}} \bar{\mm}_{\alpha,R}^{+}(E)\, \bar{\qq}_{R}(d\alpha) \\
\geq &~ \int_{Q_{R}} \mathcal{I}^{\MCP}_{0,N,T_{\alpha,R}} 
\left(\mm(E)/\mm(B_{R})\right) 
\, \bar{\qq}_{R}(d\alpha),
\end{align*}
%
%
%
where the last inequality follows from $(X_{\alpha,R}, \sfd,\bar{\mm}_{\alpha,R})$ 
being an $\MCP(0,N)$ one-dimensional space, and $T_{\alpha,R} \leq R + \diam (E)$ was introduced to obtain \eqref{E:disintnormalized}.

By \cite[Lemma 3.9]{CS19} 
$\mathcal{I}^{\MCP}_{0,N,T_{\alpha,R}}(v) \geq \mathcal{I}^{\MCP}_{0,N,R +\diam(E)}(v)$, implying the following inequality: 
$$
\mm^{+}(E) \geq \mm(B_{R}) \mathcal{I}^{\MCP}_{0,N,R +\diam(E)}\left(\frac{\mm(E)}{\mm(B_{R})}\right).
$$
%
%
%
%
%
%
%
%
We continue the chain of inequalities by 
\eqref{E:expansion}:
\begin{align*}
\mm^{+}(E) 
\geq &~ \frac{\mm(B_{R})N^{\frac{1}{N}}}{R + \diam(E)}\left( 
 \left(\frac{\mm(E)}{\mm(B_{R})}\right)^{\frac{N-1}{N}}
+
o\left(\left(\frac{\mm(E)}{\mm(B_{R})}\right)^{\frac{N-1}{N}}\right)
\right).
\end{align*}
From the hypothesis of Euclidean volume growth ($\AVR_{X} > 0$), one infers 
that $R \sim \mm(B_{R})^{\frac{1}{N}}$ for large values of $R$, hence  
we can then take the limit as $R \to \infty$ to obtain 
$$
\mm^{+}(E) \geq 
 \left(N \omega_{N}\AVR_{X}\right)^{\frac{1}{N}} \mm(E)^{\frac{N-1}{N}},
$$
proving the claim, provided $E$ is bounded. \\

{\bf Step 2.} Any $E\subset X$. \\
The general case follows by a general relaxation principle investigated in 
\cite{AmbroGigDimar}. If $(X,\sfd,\mm)$ is any \mms\ (no curvature assumptions are needed), it is shown that (see Theorem~3.6) 
for any Borel set with $A \subset X$ with $\mm(A)<\infty$
it holds true 
\begin{equation}\label{E:MP}
\inf \left\{\liminf_{h\to 0} \mm^{+}(A_{h}) \colon A_{h} \to A \right\} = Per(A),
\end{equation}
where, by definition, $A_{h} \to A$ if $\mm(A_{h} \Delta A) \to 0$, and $Per$ 
is the perimeter defined in $(X,\sfd,\mm)$; for its definition we refer to 
\cite{AmbroGigDimar} and references therein.
Inspecting the proof of \cite[Theorem 3.6]{AmbroGigDimar}, in particular the ``$\leq$'' part, one notices that one can furthermore impose $A_{h}$ to be bounded 
and still the identification \eqref{E:MP} is valid.
Hence for any $E \subset X$ with $\mm(E) < \infty$ we deduce that 
$$
Per(E) \geq 
 \left(N \omega_{N}\AVR_{X}\right)^{\frac{1}{N}} \mm(E)^{\frac{N-1}{N}}.
$$
Since $Per(E) \leq \mm^{+}(E)$, the claim follows.
\end{proof}

\smallskip
\subsection{Sharp Inequality}
As one can expect from the sharpness of the isoperimetric inequality 
for compact $\MCP(0,N)$ spaces obtained in \cite{CS19}, 
also inequality~\eqref{E:isopAVR} is sharp. 
In particular, if we fix
$a,v >0$, $N>1$, we can find a $\MCP(0,N)$ space $(X,\sfd,\mm)$, with
$\AVR_X=a$, and a subset $E\subset X$ such that
$\mm(E)=v$ and
$\mm^+(E)=(N\omega_N\AVR_X)^{\frac{1}{N}}\mm(E)^{\frac{N-1}{N}}$.
Indeed, consider the one-dimensional space $([0,\infty),|\cdot|,h\L)$,
with
\begin{equation}
  h(x)=
  \begin{dcases}
    (N\omega_N a)^{\frac{1}{N}} v^{\frac{N-1}{N}}
    & \quad\text{ if }
    x\leq \left(\frac{v}{N\omega_N a}\right)^{\frac{1}{N}},\\
    N\omega_N ax^{N-1}
    & \quad\text{ if }
    x\geq \left(\frac{v}{N\omega_N a}\right)^{\frac{1}{N}}.
  \end{dcases}
\end{equation}
It is easy to check that $h$ satisfies~\eqref{E:MCPdef3} with $K=0$
and that $\AVR_{([0,\infty),|\cdot|,h\L)}=a$.
We take $E=[0,(\frac{v}{N\omega_N a})^{\frac{1}{N}}]$, and we
trivially have $(h\L)(E)=v$ and
\begin{equation}
(h\L)^+(E)=h((\frac{v}{N\omega_Na})^{\frac{1}{N}})=(N\omega_N
a)^{\frac{1}{N}} v^{\frac{N-1}{N}},  
\end{equation}
which corresponds to equality in inequality~\eqref{E:isopAVR}.
This easy observation concludes, together with Theorem \ref{T:isoperimetricAVR}, the proof of Theorem \ref{T:main1}.
%

%
%

\end{document}